\documentclass[a4paper,11pt]{amsart}

\usepackage{amsfonts,amssymb}
\usepackage{graphicx,tikz}
\usepackage{tikz-cd} 
\usepackage{float}
\usepackage{xcolor}
\usepackage{amsmath, amsthm, amsfonts}
\usepackage{geometry}
\usepackage{verbatim}
\usepackage{hyperref}
\usepackage[capitalise]{cleveref}
\usepackage{comment}
\usepackage{enumerate}
\usepackage{amsrefs}
\usepackage[official]{eurosym}
\usepackage{graphicx}
\usepackage{nomencl}
\usepackage{amsfonts}
\usepackage{hyperref}
\usepackage{amssymb}
\usepackage{fancyhdr}
\usepackage{amscd}
\usepackage[english]{babel}

\newtheorem{theorem}{Theorem}

\newtheorem{corollary}[theorem]{Corollary}
\newtheorem{Lemma}      {Lemma} [section]
\newtheorem{Theorem}    [Lemma] {Theorem}

\newtheorem{Corollary}  [Lemma] {Corollary}
\newtheorem{Proposition}[Lemma] {Proposition}

\newtheorem{question}{Question}

\theoremstyle{definition}
\newtheorem{Definition}[Lemma]{Definition}
\newtheorem{Remark}[Lemma]{Remark}

\numberwithin{equation}{section}

\DeclareMathOperator{\Aut}{Aut}
\DeclareMathOperator{\St}{st}
\DeclareMathOperator{\Rist}{rst}
\DeclareMathOperator{\Lc}{L}
\DeclareMathOperator{\Rc}{R}
\DeclareMathOperator{\lcm}{lcm}

\newcommand{\Z}{\mathbb{Z}}
\newcommand{\N}{\mathbb{N}}
\newcommand{\F}{\mathbb{F}}

\newcommand\T{{\mathcal{T}}}
\newcommand{\FF}{\mathcal{F}}
\newcommand{\HH}{\mathcal{H}}
\newcommand{\OO}{\mathcal{O}}
\newcommand{\LL}{\mathcal{L}}

\newcommand{\RR}{\mathcal{R}}

\newcommand{\ol}{\overline}


\begin{document}

\title[Engel elements in weakly branch groups]{Engel elements in weakly branch groups}
\author[G.A.\ Fern\'andez-Alcober]{Gustavo A.\ Fern\'andez-Alcober}
\address{Department of Mathematics, University of the Basque Country UPV/EHU, Bilbao, Spain}
\email{gustavo.fernandez@ehu.eus}
\author[M.\ Noce]{Marialaura Noce}
\address{Dipartimento di Matematica, Universit\`a di Salerno, Italy;
Department of Mathematics, University of the Basque Country UPV/EHU, Bilbao, Spain}
\email{mnoce@unisa.it}
\author[G.M.\ Tracey]{Gareth M.\ Tracey}
\address{Department of Mathematical Sciences, University of Bath, Bath BA2 7AL, United Kingdom}
\email{G.M.Tracey@bath.ac.uk}

\thanks{The first two authors are supported by the Spanish Government, grant MTM2017-86802-P, partly with FEDER funds.
The first author is also supported by the Basque Government, grant IT974-16.
The second author is partially supported by the ``National Group for Algebraic and Geometric Structures, and their Applications" (GNSAGA - INdAM). 
The third author acknowledges the EPSRC (grant number 1652316) for their support.}

\maketitle

\begin{abstract}
We study properties of Engel elements in weakly branch groups, lying in the group of automorphisms of a spherically homogeneous rooted tree.
More precisely, we prove that the set of bounded left Engel elements is always trivial in weakly branch groups.
In the case of branch groups, the existence of non-trivial left Engel elements implies that these are all $p$-elements and that the group is virtually a $p$-group (and so periodic) for some prime $p$.
We also show that the set of right Engel elements of a weakly branch group is trivial under a relatively mild condition.
Also, we apply these results to well-known families of weakly branch groups, like the multi-GGS groups.
\end{abstract}


\section{Introduction}

A rapidly developing area of group theory studies the properties of \emph{branch groups}, a special kind of groups acting on spherically homogeneous rooted trees.
These groups, which were first defined by Grigorchuk at the Groups St Andrews conference in Bath in 1997, are generalizations of the famous $p$-groups constructed by Grigorchuk himself \cites{Grig0}, and by Gupta and Sidki \cite{GS}.
Despite their relatively recent introduction, branch groups have appeared in the literature in the past, without being explicitly defined.
For instance, the class of branch groups contains one of the three classes of groups in John Wilson's famous characterisation of just infinite groups \cite{JW}.
This is one of the primary reasons for their study.
Another important motivation for studying branch groups comes from the remarkable properties that some of these groups can possess, like intermediate growth, amenability, the congruence subgroup property, or providing a negative answer to the General Burnside Problem.
In this setting, one can also consider the larger family of \emph{weakly branch groups}, which preserve many of the most interesting features enjoyed by branch groups.
We refer the reader to \cref{sec:automorphisms of trees} for a quick introduction to these classes of groups.

With this motivation in mind, the purpose of this paper is to investigate Engel elements in weakly branch groups.
Given two elements $g$ and $x$ in a group $G$, we define $[g,_n x]$ for all $n\in\N\cup\{0\}$ by means of $[g,_0 x]=g$ and, for $n\ge 1$,
\[
[g,_n x] = [[g,_{n-1} x],x] = [[\hdots[[g,x],x],\hdots],x],
\qquad
\text{(where $x$ appears $n$ times).}
\]
Engel conditions in group theory have to do with the triviality of these iterated left normed commutators.
If $[g,_n x]=1$ for some $n\in\N$, we say that $x$ is \emph{Engel on $g$}, and the smallest such $n$ is the \emph{Engel degree of $x$ on $g$}.
If $x$ is Engel on all elements $g\in G$, we say that $x$ is a \emph{left Engel element} of $G$.
Observe that the Engel degree of $x$ can vary as $g$ runs over $G$, and in principle could be unbounded.
If there is a bound for the Engel degrees of $x$, i.e.\ if there exists $n\in\N$ such that
$[g,_n x]=1$ for all $g\in G$, we say that $x$ is a \emph{bounded left Engel element} of $G$.
We denote by $\Lc(G)$ and $\ol{\Lc}(G)$ the sets of left Engel elements and bounded left Engel elements of $G$, respectively.
On the other hand, if $g\in G$ is such that every $x\in G$ is Engel on $g$, we say that $g$ is a
\emph{right Engel element} of $G$.
\emph{Bounded right Engel elements} are defined in an obvious way.
We write $\Rc(G)$ and $\ol{\Rc}(G)$ for the sets of right Engel elements and bounded right Engel elements of $G$.
Observe that $\Rc(G)\subseteq \Lc(G)^{-1}$ \cite[12.3.1]{rob}, and that obviously
$\ol{\Lc}(G)\subseteq \Lc(G)$ and $\ol{\Rc}(G)\subseteq \Rc(G)$.
In particular, if $L(G)=1$ then all four Engel sets are trivial.

We say that $G$ is an \emph{Engel group} if $L(G)=G$ (or equivalently $R(G)=G$).
On the other hand, if the identity $[g,x,\overset{n}{\ldots},x]=1$ holds for all $g,x\in G$, i.e.\ if every $x\in G$ is a bounded left Engel element with a common bound for all $g\in G$ (or equivalently with the right Engel condition), then $G$ is said to be an \emph{$n$-Engel group}.

In every group $G$, the sets $\Lc(G)$, $\ol{\Lc}(G)$, $\Rc(G)$, and $\ol{\Rc}(G)$ contain some distinguished subgroups, namely the Hirsch-Plotkin radical, the Baer radical, the hypercenter and the
$\omega$-center, respectively.
In his book \textit{A Course in the Theory of Groups}, D.J.S.\ Robinson considers it one of the major goals of Engel theory to find conditions which will guarantee that these four sets of Engel elements coincide with the corresponding subgroups \cite[Section 12.3]{rob}.
For example, Baer proved that this is the case if $G$ satisfies the maximal condition
(see \cite[Satz L$^\prime$]{Baer} or \cite[12.3.7]{rob}); in particular, $\Lc(G)$ coincides with the Fitting subgroup if $G$ is finite.
However, these equalities do not hold in general.
It is then natural to ask whether $\Lc(G)$, $\ol{\Lc}(G)$, $\Rc(G)$, and $\ol{\Rc}(G)$ are always subgroups of $G$, and it was not until recently that the first counterexamples were found.
It is here where branch groups come into play in Engel theory.

Let $\mathfrak{G}$ be the first Grigorchuk group.
This is a branch group acting on the binary tree, introduced by Grigorchuk \cite{Grig0} in 1980.
In 2006 Bludov announced \cite{bludov} that the wreath product $\mathfrak{G}\wr D_8$, with the natural action of $D_8$ on $4$ points, can be generated by Engel elements but is not an Engel group.
In particular, $\Lc(\mathfrak{G})$ is not a subgroup.
This example was never published, but ten years later, Bartholdi \cite{Bartholdi} showed that
\[
\Lc(\mathfrak{G}) = \{ x\in\mathfrak{G} \mid x^2=1 \}
\]
and, as a consequence, that $\Lc(\mathfrak{G})$ is not a subgroup.
To date, it is still an open question whether $\ol{\Lc}(G)$, $\Rc(G)$ and $\ol{\Rc}(G)$ are always subgroups.

The other major question in Engel theory is whether Engel groups are locally nilpotent.
The answer is negative in general, the main example being the so-called Golod-Shafarevich groups.
These groups also provide a negative answer to the General Burnside Problem (GBP), which can be equivalently formulated as the question of whether periodic groups are locally finite.
This resemblance between Engel and Burnside problems, together with the fact that many groups answering GBP in the negative are branch groups, and Bartholdi's result on left Engel elements of the Grigorchuk group, makes it natural to study the behaviour of the Engel sets $\Lc(G)$, $\ol{\Lc}(G)$,
$\Rc(G)$ and $\ol{\Rc}(G)$ in (weakly) branch groups.
This is the specific goal that we are addressing in this paper.

Before proceeding to state our main theorems, let us mention some results in the literature regarding Engel elements in groups of automorphisms of spherically homogeneous rooted trees.
In the aforementioned paper, Bartholdi also proved that, if $G$ is the Gupta-Sidki $3$-group, then
$\Lc(G)=1$.
On the other hand, in \cite{albert}, Garreta and the first two authors proved that again $\Lc(G)=1$ if $G$ is any fractal subgroup of a Sylow pro-$p$ subgroup of the group of automorphisms of the $p$-adic tree satisfying the condition $|G':\St_G(1)'|=\infty$, and in particular if $G$ is non-abelian and has torsion-free abelianization.
Also, Tortora and the second author showed in \cite{MT} that
$\ol{\Lc}(\mathfrak{G})=\Rc(\mathfrak{G})=1$ for the Grigorchuk group $\mathfrak{G}$.
As we next see, the situation in these classes of groups generalises to a great extent to weakly branch groups, which have a tendency to have trivial Engel sets.

Our first main result reads as follows.

\begin{theorem}
\label{BranchTheoremW}
Let $G$ be a weakly branch group.
Then the following hold:
\begin{enumerate}
\item
$\ol{\Lc}(G)=1$.
\item
If the set of finite order elements of $\Lc(G)$ is non-trivial then it is a $p$-set for some prime $p$, and
the rigid stabilizer $\Rist_G(n)$ is a $p$-group for some $n\ge 1$.
\end{enumerate}
\end{theorem}

Thus even if weakly branch groups provide examples in which $\Lc(G)$ is not a subgroup, they cannot be used to obtain similar examples for $\ol{\Lc}(G)$.
This result can be interpreted in a similar vein to the fact that weakly branch groups, being residually finite, cannot provide examples of finitely generated infinite groups of finite exponent: the ``problem" in both cases is boundedness.
On the other hand, part (ii) of Theorem A raises the following question.

\begin{question}
Can a weakly branch group $G$ contain left Engel elements of infinite order?
If the answer is negative, then $\Lc(G)$ consists entirely of $p$-elements for some prime $p$.
\end{question}

If instead of weakly branch the group is actually branch, then we have the following stronger version of Theorem A.

\begin{theorem}
\label{BranchTheorem}
Let $G$ be a branch group.
If $\Lc(G)\ne 1$ then $G$ is periodic and there exists a prime $p$ such that:
\begin{enumerate}
\item
$\Lc(G)$ consists of $p$-elements.
\item
$G$ is virtually a $p$-group.
\end{enumerate}
\end{theorem}

Compare the results in Theorem B with the situation in the Grigorchuk group.
In that case, $\Lc(\mathfrak{G})$ consists of all elements of order $2$ in $\mathfrak{G}$, and
$\mathfrak{G}$ is a $2$-group.
On the other hand, the prime $p$ can be arbitrary in (i) and (ii): if $\FF_p$ is the group of $p$-finitary automorphisms of a $p$-adic tree then it is easy to see that $\Lc(\FF_p)=\FF_p$, and this is a $p$-group.
Observe however that, contrary to $\mathfrak{G}$, the group $\FF_p$ is not finitely generated.

\begin{question}
Are there any \emph{finitely generated} (weakly) branch groups for which the set $\Lc(G)$ is non-trivial and consists of $p$-elements for an \emph{odd} prime $p$?
\end{question}

In the following theorem we consider right Engel elements in weakly branch groups under a relatively mild condition.

\begin{theorem}
\label{theoremrightengel}
Let $G$ be a weakly branch group.
If the rigid stabilizer $\Rist_G(n)$ is not an Engel group for any $n\in\N$, then $\Rc(G)=1$.
\end{theorem}

\begin{question}
Is $\Rc(G)=1$ for every \emph{finitely generated} (weakly) branch group?
By Theorem C, this seems closely linked to this other question: can a \emph{finitely generated} (weakly) branch group be Engel?
\end{question}

Again, without finite generation, the group $\FF_p$ shows that the answer is negative in both cases.
Regarding the last question, observe that weakly branch groups cannot satisfy a law
(see \cite[Corollary 1.4]{abert} or \cite{leonov}) and so cannot be $n$-Engel for a fixed $n$.
Thus we are asking whether finite generation makes it impossible for them to be Engel as well.

\vspace{8pt}

As an application of Theorems A, B, and C, we get the following corollary, which provides information about Engel elements in some specific families of weakly branch groups.
The definition of these families is given either in Section \ref{sec:automorphisms of trees} or right before the proof of the corresponding result.

\begin{corollary}
Let $\T$ be a spherically homogeneous rooted tree.
Then the following hold:
\begin{enumerate}
\item
If $G$ is an infinitely iterated wreath product of finite transitive permutation groups of degree at least $2$, then $\Lc(G)=1$.
This applies in particular to the whole group of automorphisms of $\T$, and also to its Sylow pro-$p$ subgroups if $\T$ is a $p$-adic tree, where $p$ is a prime.
\item
If $\FF$ is the group of finitary automorphisms of $\T$, and there are infinitely many levels in which the number of descendants is greater than $2$, then $\Lc(\FF)=1$.
If $\T$ is a $p$-adic tree and $\FF_p$ is the group of $p$-finitary automorphisms of $\T$, then
$\ol{\Lc}(\FF_p)=1$.
\item
If $\HH$ is the Hanoi Tower group then $\Lc(\HH)=1$.
\item
If $G$ is a multi-GGS groups then $\Rc(G)=1$, and if $G$ is furthermore non-periodic then $\Lc(G)=1$.
\end{enumerate}
\end{corollary}

We conclude this introduction by indicating how the paper is organised.
In \cref{sec:automorphisms of trees} we first give some generalities about groups of automorphisms of spherically homogeneous rooted trees, with special emphasis on weakly branch groups.
Then we provide several results regarding orbits of such automorphisms that will be essential later on.
Our approach to the study of Engel elements in weakly branch groups is through the reduction to wreath products.
\cref{sec:engel in wreath} is devoted to a careful analysis of the scenarios that will arise when we apply this kind of reduction.
Finally, in \cref{sec:left engel} we prove Theorems A and B, regarding left Engel elements in weakly branch groups, and then in \cref{sec:right engel} we obtain Theorem C about right Engel elements.
The proof of the applications given in Corollary D is split between these two sections.

\vspace{8pt}

\noindent
\textit{Notation and terminology.}
If $f:X\rightarrow Y$ and $g:Y\rightarrow Z$ are two maps, we write $fg$ for their composition instead of $g\circ f$.
As usual, $S_n$ stands for the symmetric group on $n$ letters.
We denote the direct product of groups $G_1,\ldots,G_n$ by $\prod_{i=1}^n \, G_i$.
Given a group $G$, an element $g\in G$ is a \emph{$p$-element}, where $p$ is a prime, if its order is a power of $p$, and a subset $X$ of $G$ is a \emph{$p$-set} if every element of $X$ is a $p$-element.
Also, if $G$ is finite, $F(G)$ denotes the Fitting subgroup of $G$.

\section{Automorphisms of spherically homogeneous rooted trees}
\label{sec:automorphisms of trees}

In this section, we first give some notation and general facts about groups of automorphisms of a spherically homogeneous rooted tree, and more specifically, about (weakly) branch groups.
For further information on the topic one can see, for example, \cite{branch} or \cite{Grigbook}.
Then we provide some results regarding orbits of automorphisms of such trees that will be needed in the following sections.


Let $\ol{d}=\{d_n\}_{n=1}^{\infty}$ be an infinite sequence of integers greater than $1$, and let $d=d_1$.
We write $\T_{\ol{d}}$ to denote the \emph{spherically homogeneous rooted tree} corresponding to
$\ol{d}$.
This is a rooted tree where all vertices at level $n$ (i.e.\ at distance $n$ from the root) have the same number $d_{n+1}$ of immediate descendants.
If $\ol{d}$ takes the constant value $d$, we write $\T_d$ for $\T_{\ol{d}}$, and we call it the
\emph{$d$-adic tree}.
In order to ease notation, and unless it is strictly necessary to make the sequence $\ol{d}$ explicit, all throughout the paper we will simply write $\T$ to denote an arbitrary spherically homogeneous rooted tree.
Also we will write $V(\T)$ for the set of vertices of $\T$ and, for every $n\in\N$, we let $\LL_n$ be the set of all vertices on the $n$th level of $\T$.

\begin{figure}[h!]
\includegraphics[scale=0.8]{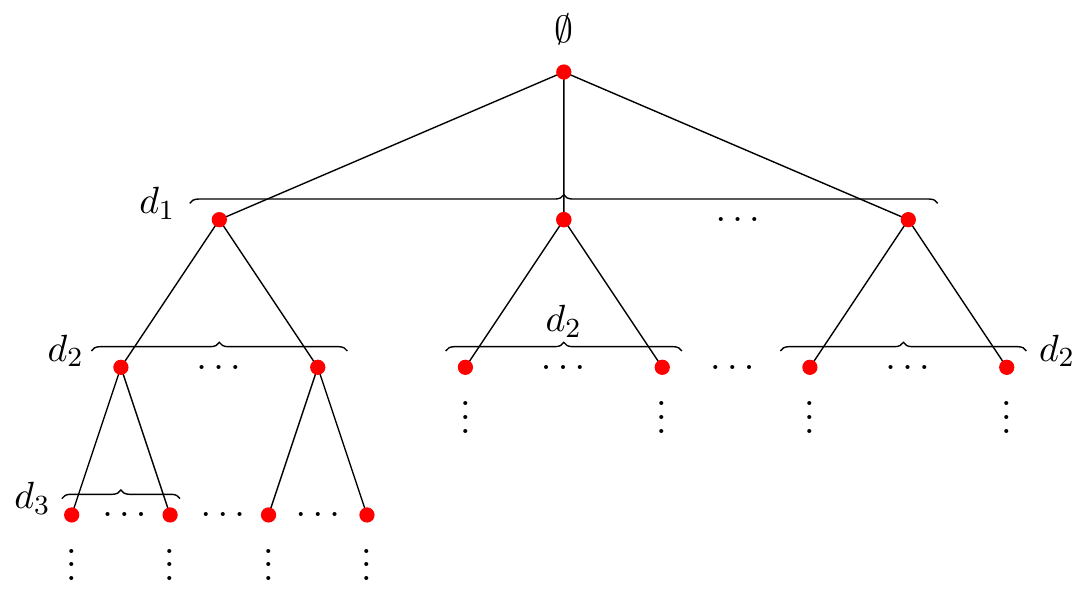}
\caption{A spherically homogeneous rooted tree}
\end{figure}

Let $\Aut\T$ be the group of automorphisms of $\T$ (i.e.\ bijective maps from $V(\T)$ to itself that preserve the root and incidence) under the operation of composition.
Every $f\in\Aut\T$ can be described by providing, at every vertex $v$ of the tree, the permutation
$f_{(v)}$ that indicates how $f$ sends the descendants of $v$ onto the descendants of $f(v)$.
This permutation is called the \emph{label} of $f$ at $v$, and if $v$ lies at level $n$ then
$f_{(v)}\in S_{d_{n+1}}$.
The collection of all labels of $f$ constitutes the \emph{portrait} of $f$, and there is a one-to-one correspondence between automorphisms of $\T$ and portraits.
An automorphism of $\T$ is called \emph{finitary} if it has finitely many non-trivial labels in its portrait.
Finitary automorphisms form a locally finite subgroup $\FF$ of $\Aut\T$.
If $\T$ is a $p$-adic tree for a prime $p$ and we fix a $p$-cycle $\sigma$ in $S_p$, the group $\FF_p$ of finitary automorphisms whose labels are all powers of $\sigma$ constitute a subgroup of $\FF$.
We call this the \emph{group of $p$-finitary automorphisms} of $\T$.
(We give no reference to $\sigma$, since different choices of the $p$-cycle give rise to isomorphic groups.)
Observe that $\FF_p$ is locally a finite $p$-group.

We write $\St(v)$ for the stabilizer of $v\in V(\T)$ in $\Aut\T$ and $\St(n)$ for the pointwise stabilizer of
$\LL_n$, i.e.\
\[
\St(n) = \cap_{v\in\LL_n} \, \St(v).
\]
The latter is a normal subgroup of finite index of $\Aut\T$.
The factor group $\Aut\T/\St(n)$ is naturally isomorphic to the automorphism group of the finite tree consisting of all levels of $\T$ up to (and including) the $n$th level.
Then $\Aut\T$ is isomorphic to the inverse limit of these finite groups, and is so a profinite group.
Also, we have
\[
\Aut\T
\cong
\cdots ( S_{d_n} \wr ( \cdots ( S_{d_3}\wr (S_{d_2}\wr S_{d_1})) \cdots )) \cdots,
\]
where the iterated wreath product is permutational at every step.
If we consider the $p$-adic tree $\T_p$, where $p$ is a prime, and we consider a fixed $p$-cycle
$\sigma\in S_p$, then the set $\Gamma_p$ of all automorphisms of $\T$ with labels in
$\langle \sigma \rangle$ is a Sylow pro-$p$ subgroup of $\Aut\T_p$.
We say that $\Gamma_p$ is a \emph{standard Sylow pro-$p$ subgroup} of $\Aut\T_p$.
Observe that
\[
\Gamma_p
\cong
\cdots ( C_p \wr ( \cdots ( C_p\wr (C_p\wr C_p)) \cdots )) \cdots.
\]

Let $\T_v$ be the subtree hanging from the vertex $v$ of the tree.
We have $\T_u \cong \T_v$ for any two vertices $u$, $v$ on the same level, and we denote by
$\T_{\langle n \rangle}$ any tree isomorphic to a subtree with root in $\LL_n$.
If $s$ denotes the shift operator that erases the first term of a sequence, then
$\T_{\langle n \rangle}$ is isomorphic to the spherically homogeneous tree defined by the sequence
$s^n(\,\ol{d}\,)$.

Every $f\in\Aut\T$ naturally induces a bijection between $\T_v$ and $\T_{f(v)}$ which, under the identification of these trees with $\T_{\langle n \rangle}$, defines an automorphism $f_v$ of
$\T_{\langle n \rangle}$.
This is called the \emph{section} of $f$ at $v$.
Sections satisfy the following rules, for all $f,g\in\Aut\T$ and $v\in V(\T)$:
\[
(fg)_v = f_v g_{f(v)}
\]
and
\[
(f^g)_{g(v)} = (g_v)^{-1} f_v \; g_{f(v)}.
\]
As a consequence, if $f\in\St(v)$, we get
\begin{equation}
\label{section of conjugate}
(f^g)_{g(v)} = (f_v)^{g_v}.
\end{equation}

If $f$ fixes the vertex $v$, then the section $f_v$ is nothing but the restriction of $f$ to $\T_v$.
The assignment $f\mapsto f_v$ induces a homomorphism
$\psi_v \colon \St(v) \rightarrow \Aut\T_{\langle n \rangle}$, and the map
\[
\begin{matrix}
\psi_n & \colon & \St(n) & \longrightarrow
& \Aut\T_{\langle n \rangle} \times \overset{d_1\ldots d_n}{\cdots} \times \Aut\T_{\langle n \rangle}
\\[5pt]
& & f & \longmapsto & \left( f_v \right)_{v\in \LL_n}.
\end{matrix}
\]
is an isomorphism.
If $n=1$ we simply write $\psi$ for $\psi_1$.
In the case of a $d$-adic tree, we get
\[
\St(n) \cong \Aut\T \times \overset{d^n}{\cdots} \times \Aut\T.
\]

Observe also that $\Aut\T$ splits over $\St(n)$ for every $n\ge 1$.
One can take as a complement the subgroup
\begin{align*}
H_n
&=
\{ f\in\Aut\T \mid f_v=1 \text{ for all $v\in\LL_n$} \}
\\
&=
\{ f\in\Aut\T \mid f_{(v)}=1 \text{ for all $v\in\cup_{i\ge n} \, \LL_i$} \}.
\end{align*}
The automorphisms in $H_1$ are called \emph{rooted automorphisms} of $\T$.
They act on $\T$ by permuting rigidly the subtrees hanging from the root according to some permutation of $S_d$.
Every automorphism $f\in\Aut\T$ can be uniquely written in the form $gh$, where $g\in\St(1)$ and $h$ is rooted.
If $\psi(g)=(g_1,\ldots,g_d)$ and $h$ corresponds to a permutation $\sigma\in S_d$, we use the following
shorthand notation to denote $f$:
\[
f = (g_1,\ldots,g_d) \sigma.
\]

Now let $G$ be a subgroup of $\Aut\T$.
We set $\St_G(n)=\St(n)\cap G$ for all $n\ge 1$.
If $v$ is a vertex of $\T$, the \emph{rigid stabilizer} of $v$ in $G$ is defined as follows:
\[
\Rist_G(v) = \{ g\in G \mid \text{$g(u)=u$ for all $u$ lying outside $\T_v$} \}.
\]
If $V$ is a set of vertices, all lying on the same level of $\T$, we set
\[
\Rist_G(V) = \langle \Rist_G(v) \mid v\in V \rangle,
\]
the rigid stabilizer of $V$ in $G$.
It turns out that
\[
\Rist_G(V) = \prod_{v\in V} \, \Rist_G(v).
\]
We write $\Rist_G(n)$ for the rigid stabilizer of $\LL_n$, and call it the \emph{$n$th rigid stabilizer}
of $G$.
It is the direct product of the rigid stabilizers of all vertices of $\LL_n$, and it is the largest ``geometrical" direct product inside $\St_G(n)$, in the sense that a subgroup $H$ of $\St_G(n)$ satisfies
\[
\psi_n(H) = \prod_{v\in \LL_n} \, H_v
\]
with $H_v\le \Aut\T_v$ if and only if $H\le\Rist_G(n)$.
Obviously, if $G$ is the whole of $\Aut\T$ then the $n$th rigid stabilizer coincides with the $n$th level stabilizer.
However, this is not usually the case for arbitrary subgroups of $\Aut\T$.

By \eqref{section of conjugate}, we have
\begin{equation}
\label{conjugate of rist}
\Rist_G(v)^g = \Rist_G(g(v))
\end{equation}
for every $v\in V(\T)$ and $g\in G$.
Thus if $G$ is \emph{spherically transitive}, i.e.\ if $G$ acts transitively on every $\LL_n$, then each level rigid stabilizer $\Rist_G(n)$ is a direct product of isomorphic subgroups for all $n\in\N$.
We are now ready to introduce the class of groups that are the object of our study.

\begin{Definition}
Let $G$ be a spherically transitive subgroup of $\Aut\T$.
Then:
\begin{enumerate}[(a)]
\item
If $|G:\Rist_G(n)|<\infty$ for all $n\in\N$, we say that $G$ is a \emph{branch group}.
\item
If $\Rist_G(n)\ne 1$ for all $n\in\N$, we say that $G$ is a \emph{weakly branch group}.
\end{enumerate}
\end{Definition}

Notice that all rigid level stabilizers in a weakly branch group are infinite.
Also, since spherically transitive groups are infinite, branch groups are obviously weakly branch.

After this quick introduction to groups of automorphisms of a spherically homogeneous rooted tree, we start developing the tools that we will use in the proof of Theorems A, B, and C.
A key ingredient in our approach to Engel problems in weakly branch groups is the reduction of the action of an automorphism $f$ from the whole tree to one or several ``reduced trees" determined by some special orbits of $f$ on $V(\T)$.
For this reason, we start by describing some properties of orbits of automorphisms of $\T$.

\begin{Definition}
If $f\in\Aut\T$ and $v\in V(\T)$, the \emph{$f$-orbit} of $v$ is the orbit of $v$ under the action of
$\langle f \rangle$ on $V(\T)$, i.e.\ the set $\{ f^i(v) \mid i\in\Z \}$.
The $f$-orbit is \emph{trivial} if it consists of only one vertex, that is, if $f(v)=v$.
\end{Definition}

In the statement of the following lemma, we consider the least common multiple of an unbounded family of positive integers to be infinity.

\begin{Lemma}
\label{lcm of orbit lengths}
Let $f\in\Aut\T$ and, for every vertex $v\in V(\T)$, let $\OO_v$ be the $f$-orbit of $v$.
Then the following hold:
\begin{enumerate}
\item
If $w$ is a descendant of $v$, then $|\OO_v|$ divides $|\OO_w|$.
\item
$|f|=\lcm( |\OO_v| \mid v\in V(\T) )$.
\item
If $|f|$ is finite then there exists a finite subset $V$ of $V(\T)$ satisfying that
$|f|=\lcm( |\OO_v| \mid v\in V )$ and that, whenever $w$ is a descendant of a vertex $v\in V$,
we have $|\OO_w|=|\OO_v|$.
Also if $f$ is non-trivial then all the orbits $\OO_v$ with $v\in V$ are non-trivial.
Furthermore, $V$ can be chosen to lie in $\LL_n$ for some $n$.
\end{enumerate}
\end{Lemma}

\begin{proof}
(i)
This is obvious by the orbit-stabilizer theorem, since $\St(w)\subseteq \St(v)$.

(ii)
Set $H=\langle f \rangle$.
Then $|\OO_v|=|H/\St_H(v)|$ for all $v\in V(\T)$.
The natural map $\varphi$ from $H$ to the cartesian product of finite groups
$\prod_{v\in V(\T)} \, H/\St_H(v)$ is injective, since the intersection of all vertex stabilizers is trivial.
Consequently
\[
|f| = |\varphi(f)| = \lcm( |f\St_H(v)| \mid v\in V(T) ) = \lcm( |H/\St_H(v)| \mid v\in V(T) ),
\]
which proves the result.

(iii)
Let $L=\{|\OO_v| \mid v\in V(\T) \}$.
If $|f|$ is finite then, by (ii), it can be achieved as the least common multiple of a finite subset of $L$.
Let $k$ be the minimum cardinality of such a subset and let
\[
\mathcal{S} = \{ S\subseteq L \mid \text{$|S|=k$ and $\lcm(S)=|f|$} \}.
\]
Observe that $\mathcal{S}$ is a finite set.

We introduce a relation $\le_{\text{d}}$ in $\mathcal{S}$ by letting $S\le_{\text{d}} T$ if there exists a bijection $\alpha:S\rightarrow T$ such that $s\mid \alpha(s)$ for all $s\in S$.
By (i), this models the situation when we pass from the orbits of a set of vertices to the orbits of a set of descendants of those vertices.
We claim that $\le_\text{d}$ is an order relation in $\mathcal{S}$.
Obviously, only antisymmetry needs to be checked.
Assume that $\alpha:S\rightarrow T$ and $\beta:T\rightarrow S$ are such that $s\mid \alpha(s)$ and $t\mid \beta(t)$ for all $s\in S$ and $t\in T$.
Then $s$ divides $\beta(\alpha(s))$ and, if they are not equal, we get $\lcm(S\smallsetminus \{s\})=|f|$.
This is contrary to the minimality condition imposed on $k$.
Thus $\beta(\alpha(s))=s$ and, since $s\mid \alpha(s)$ and $\alpha(s)\mid \beta(\alpha(s))$, we obtain
that $\alpha(s)=s$ for all $s\in S$.
We conclude that $S=T$, which proves antisymmetry of $\le_{\text{d}}$.

Now choose $S$ in $\mathcal{S}$ that is maximal with respect to the order $\le_{\text{d}}$, and let
$V=\{v_1,\ldots,v_k\}\subseteq V(\T)$ be such that $S=\{|\OO_{v_1}|,\ldots,|\OO_{v_k}|\}$.
Consider an arbitrary set of vertices $W=\{w_1,\ldots,w_k\}$, where each $w_i$ is a descendant of $v_i$, and let $T=\{|\OO_{w_1}|,\ldots,|\OO_{w_k}|\}$.
Then $S\le_{\text{d}} T$ and, by the maximality of $S$, we have $S=T$.
This implies that $|\OO_{w_i}|=|\OO_{v_i}|$ for all $i=1,\ldots,k$.
Observe also that the minimality of $k$ implies that, if $f$ is non-trivial, no orbit $\OO_v$ with $v\in V$ is of length $1$.
Hence $V$ satisfies the properties stated in (iii).

Finally, observe that also the set $W$ satisfies the required properties.
Thus by considering, for a suitable $n$, a subset of $\LL_n$ consisting of one descendant of each vertex in $V$, we may assume that $V\subseteq\LL_n$.
\end{proof}

Vertices and orbits as in part (iii) of the previous lemma will play a fundamental role in the rest of the paper, and it is convenient to introduce some terminology.

\begin{Definition}
Let $f\in\Aut\T$ and let $\OO$ be an $f$-orbit.
We say that $\OO$ is \emph{totally splitting} if for every
descendant $w$ of a vertex $v\in\OO$, the length of the $f$-orbit of $w$ is equal to $|\OO|$.
\end{Definition}

Equivalently, an $f$-orbit $\OO$ is totally splitting when the set of descendants of the vertices in $\OO$ at every level of the tree splits into the maximum possible number of $f$-orbits.

\begin{Definition}
Let $f\in\Aut\T$ be an automorphism of finite order.
If $V$ is a finite set of vertices satisfying the conditions in (iii) of \cref{lcm of orbit lengths}, all of them lying on the same level of $\T$, we say that $V$ is a \emph{fundamental system of vertices} for $f$.
\end{Definition}

Next we give a sufficient condition for two automorphisms of $\T$ to generate a wreath product.

\begin{Lemma}
\label{wreath}
Let $f\in\Aut\T$ be an automorphism of finite order $m$, and assume that the $f$-orbit of a vertex
$v\in V(\T)$ has length $m$.
Then for every $g\in\Rist(v)$, the subgroup $\langle g,f \rangle$ of $\Aut\T$ is isomorphic to the regular wreath product $\langle g \rangle \wr \langle f \rangle$.
\end{Lemma}

\begin{proof}
Let $\OO$ be the $f$-orbit of $v$.
Since $|\OO|=|f|$, we have $\langle f \rangle \cap \St(v)=1$.
As a consequence, if $v$ lies at level $n$ of the tree, also $\langle f \rangle \cap \St(n)=1$ and
\begin{equation}
\label{<g,f> semidirect}
\langle g,f \rangle = \langle f \rangle \, \langle g,g^f,\ldots,g^{f^{m-1}} \rangle
= \langle f \rangle \ltimes \langle g,g^f,\ldots,g^{f^{m-1}} \rangle,
\end{equation}
since $g\in\Rist(v)$ implies that $\langle g,g^f,\ldots,g^{f^{m-1}} \rangle \subseteq \St(n)$.

Now set $v_i=f^i(v)$ for all $i\in\Z$, so that $\OO=\{v_0,v_1,\ldots,v_{m-1}\}$.
Since $g\in\Rist_G(v)$, from \eqref{conjugate of rist} we get $g^{f^i}\in\Rist_G(v_i)$ for all $i=0,\ldots,m-1$, and then
\[
\langle g^{f^i} \rangle \cap \langle g,g^f,\ldots,g^{f^{i-1}} \rangle
\subseteq \Rist_G(v_i)\cap \Rist_G(\{v_1,\ldots,v_{i-1}\}) = 1.
\]
Also $[g^{f^i},g^{f^j}]=1$ for every $i,j\in\{0,\ldots,m-1\}$.
It follows that
\[
\langle g,g^f,\ldots,g^{f^{m-1}} \rangle
=
\langle g \rangle \times \langle g^f \rangle \times \cdots \times \langle g^{f^{m-1}} \rangle,
\]
and since $g^{f^m}=g$, we conclude from \eqref{<g,f> semidirect} that
$\langle g,f \rangle \cong \langle g \rangle \wr \langle f \rangle$.
\end{proof}

The result in \cref{wreath} raises the question of whether an automorphism $f\in\Aut\T$ of finite order $m$ must have a regular orbit on $V(\T)$, i.e.\ an orbit of length $m$.
This is clearly the case if $m$ is a prime power, by (ii) of \cref{lcm of orbit lengths}, but it usually fails otherwise.
Indeed, one can consider for example a rooted automorphism corresponding to a permutation whose order is strictly bigger than the lengths of its disjoint cycles.
However, as we see in \cref{f->fi} below, it is always possible to derive a collection of automorphisms $f_i$ from $f$, acting not on $\T$ but on some other rooted trees $\RR_i$ obtained from $\T$, and having the property that every $f_i$ has a regular orbit on $V(\RR_i)$.
These automorphisms $f_i$ will allow us to study Engel conditions regarding $f$ by using \cref{wreath}.

As we will see, \cref{f->fi} is essentially a reformulation of (iii) of \cref{lcm of orbit lengths}.
Before proceeding we need to introduce the concept of reduced tree.
Note that reduced trees are somehow related to the trees obtained by deletion of layers
defined by Grigorchuk and Wilson in \cite{uniqueness}.

\begin{Definition}
Let $V$ be a subset of vertices of $\T$, all lying on the same level $n$.
We define the \emph{reduced tree} of $\T$ at $V$, denoted by $\RR(V)$, as the rooted tree consisting of the subtrees $\T_v$ for $v\in V$, all connected to a common root.
In other words, the set of vertices of $\RR(V)$ is
\[
\{ \emptyset \} \cup \{ vw \mid v\in V,\ w\in \T_{s^{n+1}(\, \overline d \,)} \},
\]
where as before $s$ denotes the shift operator on sequences.
\end{Definition}

For example, in the following figure, we consider the rooted automorphism $f$ of the ternary tree $\T_3$ corresponding to the permutation $(1\ 2\ 3)$ and we show in red the reduced tree at the orbit of the vertex $13$:

\begin{figure}[h!]
\includegraphics[scale=0.6]{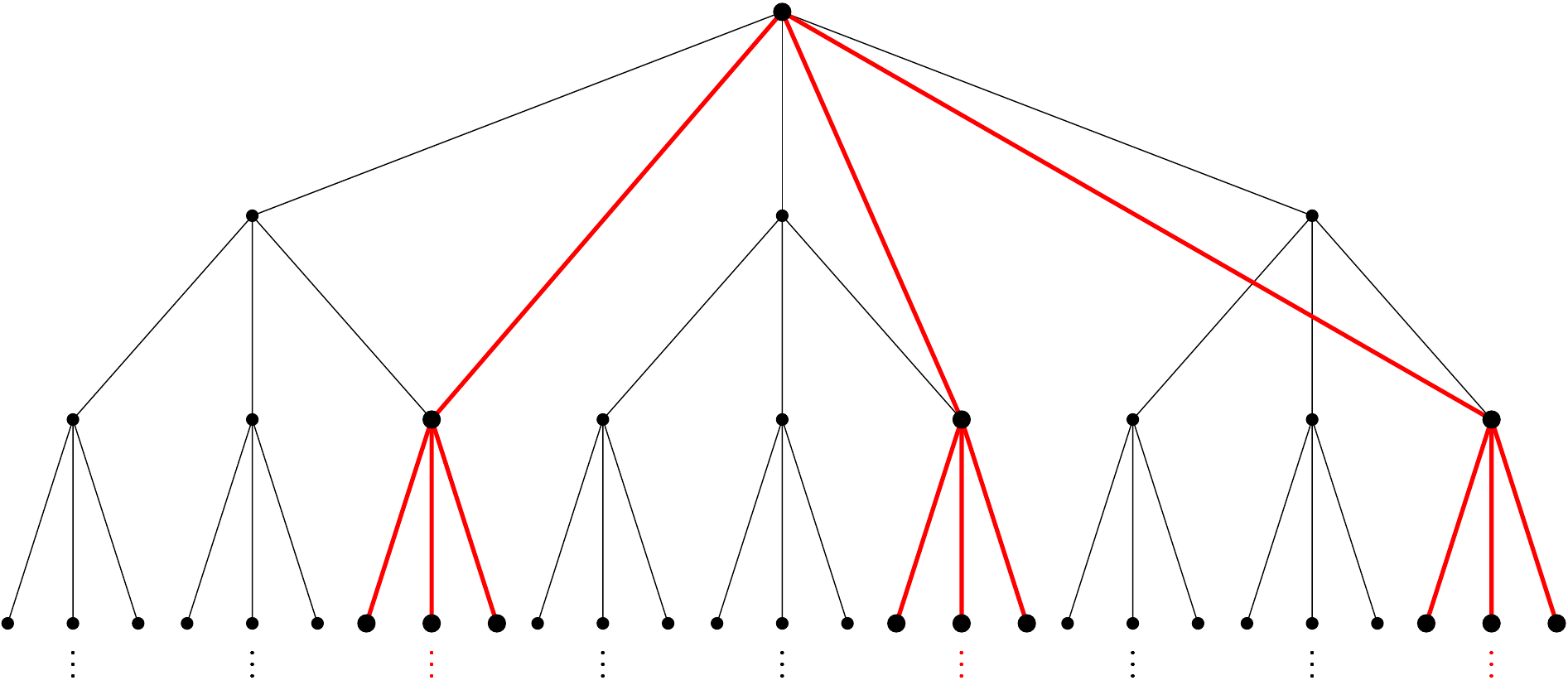}
\caption{An $f$-orbit and its corresponding reduced tree}
\end{figure}

Every $f\in\Aut\T$ such that $f(V)=V$ induces by restriction an automorphism $f_V\in \Aut\RR(V)$.
Clearly, the map $\Phi_V:f\longmapsto f_V$ is a homomorphism of groups.
The effect of $\Phi_V$ is to focus on the action of $f$ only on the subtrees $\T_v$ with $v\in V$, so to speak.
We will use reduced trees mainly in the case where $V$ is an orbit of $f$.

\begin{Remark}
\label{induced regular orbit}
If $v$ is a vertex of the reduced tree $\RR(V)$ and $f\in\Aut\T$ is such that $f(V)=V$, then the
$f_V$-orbit of $v$ coincides with the $f$-orbit of $v$ as a vertex in $V(\T)$.
In particular, if $\OO$ is a totally splitting $f$-orbit and we consider the induced automorphism $x=\Phi_{\OO}(f)$ of $\RR(\OO)$, then (ii) of \cref{lcm of orbit lengths} implies that 
$|x|=|\OO|$.
In other words, $\OO$ is a regular orbit of $x$ in $\RR(\OO)$.
\end{Remark}

Given a subgroup $G$ of $\Aut\T$, we write $G_V$ for the image of the setwise stabilizer of
$V$ in $G$ under the homomorphism $\Phi_V$.
In other words,
\[
G_V = \{ f_V \mid \text{$f\in G$ and $f(V)=V$} \}.
\]
Then $G_V$ is a subgroup of $\Aut \RR(V)$, and for every vertex $v\in V$ we have
$\Phi_V(\Rist_G(v))\subseteq \Rist_{G_V}(v)$ (the inclusion can be proper, since there can be automorphisms in $G$ whose action is trivial on $\T_w$ for every $w\ne v$ with $w\in V$, but non-trivial for some $w\not\in V$).

On the other hand, if $f\in G$ stabilizes the set $V$ and $x=\Phi_V(f)$ is the induced automorphism of
$\RR(V)$, then $f\in \Lc(G)$ or $f\in\ol{\Lc}(G)$ imply that $x\in\Lc(H)$ or $x\in\ol{\Lc}(H)$, respectively.
In particular, by choosing $V$ to be an $f$-orbit, this will allow us to transfer the analysis of a given Engel element in a subgroup of $\Aut\T$ to a more restricted situation where, for example, the Engel element acts transitively on the first level of the tree.

Actually the most convenient strategy is to reduce the tree to non-trivial totally splitting $f$-orbits, since the induced automorphisms will then have regular orbits.
More precisely, we will rely on the following lemma, which is basically a rephrasing of part of
\cref{lcm of orbit lengths} in the language of reduced trees.

\begin{Lemma}
\label{f->fi}
Let $f\in\Aut\T$ be an automorphism of finite order $m>1$ and let $\{v_1,\ldots,v_k\}$ be a fundamental system of vertices for $f$.
For every $i=1,\ldots,k$, let $\OO_i$ be the $f$-orbit of $v_i$, set $\RR_i=\RR(\OO_i)$, and let $f_i$ be the automorphism of $\RR_i$ induced by $f$.
Then the following hold:
\begin{enumerate}
\item
$\lcm(|\OO_1|,\ldots,|\OO_k|)=m$. 
\item
$\OO_i$ is a non-trivial totally splitting $f$-orbit for every $i=1,\ldots,k$.
\item
$|f_i|=|\OO_i|$ for every $i=1,\ldots,k$.
\end{enumerate} 
\end{Lemma}

\begin{proof}
The first two items follow from (iii) of \cref{lcm of orbit lengths}, and (iii) from \cref{induced regular orbit}.
\end{proof}

\section{Some properties of Engel elements in wreath products}
\label{sec:engel in wreath}

In this section we prove several results regarding Engel elements in wreath products.
These will provide the basis for the proof of the main theorems in this paper, which will be addressed in Sections \ref{sec:left engel} and \ref{sec:right engel}.

We start by studying left Engel elements lying outside the base group of a regular wreath product of two cyclic groups.
To this purpose, we rely on the paper \cite{Liebeck} by Liebeck.

\begin{Lemma}
\label{Lieb}
Let $X=\langle x\rangle$ and $Y=\langle y\rangle$ be two non-trivial cyclic groups, where $X$ is finite, and let $W=Y\wr X$ be the corresponding regular wreath product.
If $x\in \Lc(W)$ then $X$ and $Y$ are finite $p$-groups for some prime $p$.
Furthermore, the Engel degree of $x$ on $g=(y,1,\ldots,1)$ is equal to
\[
|x|+\frac{1}{p}(\log_p |y|-1)(p-1)|x|.
\]
\end{Lemma}

\begin{proof}
Let $m$ be the order of $x$, and let $p$ be an arbitrary prime divisor of $m$.
Also, write $d$ for the Engel degree of $x$ on $g=(y,1,\ldots,1)$.

First of all, suppose that $Y$ is finite.
Then $W$ is finite and, by Baer's theorem mentioned in the introduction, $x$ lies in the Fitting subgroup $F(W)$.
We claim that $Y$ is then a $p$-group.
To this purpose, assume that $|Y|$ is divisible by a prime $q\ne p$, and let
$Z=\langle z \rangle\ne 1$ be the subgroup of $Y$ of order $q$.
Consider the direct product $Z^X$ inside the base group of $W$.
Since $Z^X$ is abelian and normal in $W$, it lies in $F(W)$.
Now $Z^X$ is a $q$-group and $x_p=x^{m/p}$ is a $p$-element, and both lie in the nilpotent group
$F(W)$.
It follows that $x_p$ centralizes $Z^X$, which is clearly a contradiction, since $x_p$ does not commute
with $(z,1,\ldots,1)$.
This proves the claim, and since this property holds for every prime divisor of $m$, it also follows that $X$ is a $p$-group.
Observe that, since both $X$ and $Y$ are finite $p$-groups, the proof of Theorem 5.1 of \cite{Liebeck} yields that
\begin{equation}
\label{bound engel degree}
d = m+\frac{1}{p}(\log_p |y|-1)(p-1)m
\end{equation}
in this case.

Now it is easy to see that $Y$ cannot be infinite.
For a contradiction, suppose that $Y$ is infinite and consider a prime $q$ different from $p$.
Then the wreath product $W_q=(Y/Y^q)\wr X$ can be seen as a factor group of $W$, and
so $x$ is a left Engel element in $W_q$.
Since $|Y/Y^q|=q$ and $p$ divides $|X|$, we get a contradiction with the previous paragraph.
\end{proof}

Now we digress from Engel elements for a moment, but still working with wreath products of cyclic groups, in order to prove that rigid stabilizers of weakly branch groups are not only infinite, but have infinite exponent (\cref{rist infinite exponent} below).

\begin{Lemma}
\label{order in wreath}
Let $X=\langle x\rangle$ and $Y=\langle y\rangle$ be two finite cyclic groups, where $Y$ is non-trivial, and let $W=Y\wr X$ be the corresponding regular wreath product.
If $g=(y,1,\ldots,1)$ then $|xg|>|x|$.
\end{Lemma}

\begin{proof}
Set $m=|x|$ and let $n\in\{1,\ldots,m\}$ be arbitrary.
Then
\[
\textstyle
{(xg)^n = x^n g_1^n g_2^{\binom{n}{2}}\ldots g_{n-1}^{\binom{n}{n-1}} g_n},
\]
where $g_1=g$ and $g_i=[g,x,\overset{i-1}{\dots},x]$ for every $i=2,\ldots,n$.
Now observe that each $g_i$ is of the form
\[
g_i=(\ast,\ldots,\ast,y,1,\ldots,1),
\]
where we use $\ast$ to denote unspecified powers of $y$, and $y$ occupies the $i$th position.
It follows that
\[
(xg)^n = x^n (\ast,\ldots,\ast,y,1,\ldots,1),
\]
where $y$ appears at the $n$th position.
In particular, $(xg)^n\ne 1$ for $1\le n\le m$, and consequently $|xg|>m$, as desired.
\end{proof}

\begin{Proposition}
\label{rist infinite exponent}
Let $G$ be a weakly branch group.
Then the exponent of $\Rist_G(n)$ is infinite for every $n\in\N$.
\end{Proposition}

\begin{proof}
By way of contradiction, assume that $\Rist_G(n)$ has finite exponent.
Thus $\Rist_G(n)$ is periodic and there is a bound for the orders of its elements.
For every $k\ge n$, let $\pi_k$ be the (finite) set of prime divisors of the orders of the elements of
$\Rist_G(k)$.
Then $\{\pi_k\}_{k\ge n}$ is a decreasing sequence of non-empty finite sets and consequently their intersection is also non-empty.
Let $p$ be a prime in $\cap_{k\ge n} \, \pi_k$.
  
Consider a $p$-element $f\in\Rist_G(n)$ of maximum order, say $m$.
Since the order of $f$ is the least common multiple of the orders of the components of $\psi_n(f)$, we may assume without loss of generality that $f\in\Rist_G(u)$ for some vertex $u$ of the $n$th level.
By (ii) of \cref{lcm of orbit lengths}, there is a vertex $v$ in the tree $\T$ such that the $f$-orbit of $v$ has length $m$.
Of course, $v$ must be a descendant of $u$.
Now the choice of $p$ allows us to consider a non-trivial $p$-element $g$ in $\Rist_G(v)$.
Set $H=\langle g,f \rangle$.
By \cref{wreath}, we have $H\cong \langle g \rangle \wr \langle f \rangle$.
In particular, $H$ is a finite $p$-group.
On the other hand, by \cref{order in wreath}, $H$ contains an element of order greater than $m$.
This contradicts the choice of $m$, since $H\subseteq \Rist_G(n)$.
\end{proof}

Now we continue with our analysis of Engel elements in some wreath products.
Before proceeding, we introduce some further notation.
If $G$ is a group and $S\subseteq G$, we write $\Lc_G(S)$ to denote the set of all $x\in G$ that are left Engel elements on every element of $S$, that is, such that for all $s\in S$ there exists $n=n(s,x)$ such that $[s,_n x]=1$.
We define the set $\ol{\Lc}_G(S)$ in the obvious way, and if $x\in\ol{\Lc}_G(S)$ then the Engel degree of $x$ on $S$ is the maximum of the Engel degrees of $x$ on the elements of $S$.

\begin{Lemma}
\label{Comm}
Let $W=Y\wr X$ be a regular wreath product of two non-trivial groups, where $X$ is finite cyclic of order $n$, and let $\pi:W\rightarrow X$ be the natural projection.
Assume that $D=D_1\times \cdots \times D_n\ne 1$ is a subgroup of the base group of $W$, and that $w\in W$ is such that $\pi(w)$ is a generator of $X$.
Then the following hold:
\begin{enumerate}
\item
If $w\in\ol{\Lc}_W(D)$ has Engel degree $d$ on $D$ then $d\ge n$.
\item
If $w\in\Lc_W(D)$ then $C_D(w)$ is periodic.
\end{enumerate}
\end{Lemma}

\begin{proof}
Write $w=(y_1,\ldots,y_n)x$, where $y_i\in Y$ and $x$ generates $X$.
We may assume that $x$ permutes the components of the base group according to the cycle
$(1\ 2\ \ldots \ n)$.

(i)
Without loss of generality, we may assume that $D_1\ne 1$.
Choose a non-trivial element $g=(y,1,\ldots,1)\in D$ and let $1\le i\le n-1$.
One can easily check by induction on $i$ that
\[
[g,_i w] = (y^{(-1)^i},\ldots,y^{y_1\ldots y_i},1,\ldots,1),
\]
where the last non-trivial component is in position $i+1$.
It follows that $[g,_{n-1} w]\ne 1$ and $d\ge n$.

(ii)
By contradiction, assume that $h=(z_1,\ldots,z_n)\in C_D(w)$ is of infinite order.
For notational convenience, set $z_0=z_n$ and $y_0=y_n$.
Then from the condition $h=h^w$ we get $z_i=z_{i-1}^{y_{i-1}}$ for all $i=1,\ldots,n$.
Hence all components of $h$ are conjugate and they are all of infinite order.

Now let $g=(z_1,1,\ldots,1)\in D$.
For every $k\ge 0$, let us write $[g,_k w]=(z_{k,1},\ldots,z_{k,n})$ and, as before, set $z_{k,0}=z_{k,n}$.
We claim that the following hold for every $k\ge 0$:
\begin{enumerate}
\item[(a)]
$z_{k,i}\in\langle z_i\rangle$ for every $i=1,\ldots,n$.
\item[(b)]
If we write $z_{k,i}=z_i^{m_{k,i}}$, then there exists $i\in\{1,\ldots,n\}$ such that $m_{k,i}\ne m_{k,i-1}$.
\end{enumerate}
We argue by induction on $k$.
The result is obvious for $k=0$, so assume $k\ge 1$ and that the claim is true for values less than $k$.
Since $[g,_k w] = [g,_{k-1} w]^{-1} [g,_{k-1} w]^w$, it follows that
\[
z_{k,i} = z_{k-1,i}^{-1} \, z_{k-1,i-1}^{y_{i-1}} = z_i^{-m_{k-1,i}} \, (z_{i-1}^{y_{i-1}})^{m_{k-1,i-1}}
= z_i^{m_{k-1,i-1}-m_{k-1,i}}
\]
for all $i=1,\ldots,n$.
This proves (a) and, if (b) does not hold, then
\[
m_{k-1,1}-m_{k-1,2} = m_{k-1,2}-m_{k-1,3} = \cdots = m_{k-1,n-1}-m_{k-1,n} = m_{k-1,n}-m_{k-1,1}.
\]
Now the sum of the $n-1$ first terms in this chain of equalities is the same as $n-1$ times the last one, i.e.\
\[
m_{k-1,1}-m_{k-1,n} = (n-1) (m_{k-1,n}-m_{k-1,1}).
\]
From this, it readily follows that
\[
m_{k-1,1} = m_{k-1,2} = m_{k-1,3} = \cdots = m_{k-1,n},
\]
which is contrary to the induction hypothesis.

Finally, observe that (b) above implies that $m_{k,i}$ and $m_{k,i-1}$ cannot both be zero.
Since $z_i$ and $z_{i-1}$ are of infinite order, we conclude that $[g,_k w]\ne 1$ for all $k\ge 1$ and consequently $w\not\in\Lc_W(D)$.
This contradiction completes the proof.
\end{proof}

\section{Left Engel elements in weakly branch groups}
\label{sec:left engel}

At this point, we can start combining all the machinery developed in Sections \ref{sec:automorphisms of trees} and \ref{sec:engel in wreath} in order to prove the main results of this paper.
In this section we consider left Engel elements.
The following is an expanded version of \cref{BranchTheoremW}.

\begin{Theorem}
\label{weakly branch case}
Let $G$ be a subgroup of $\Aut\T$ in which all rigid vertex stabilizers are non-trivial.
Then:
\begin{enumerate}
\item
If $f$ is a non-trivial left Engel element of finite order, and $\OO$ is a non-trivial totally splitting $f$-orbit, then for some prime number $p$ the length of $\OO$ is a $p$-power and $\Rist_G(\OO)$ is a
$p$-subgroup.
\end{enumerate}
If $G$ is furthermore weakly branch, then:
\begin{enumerate}
\setcounter{enumi}{1}
\item
If the set of finite order elements of $\Lc(G)$ is non-trivial then it is a $p$-set for some prime $p$, and
$\Rist_G(n)$ is a $p$-group for some $n\geq 1$.
\vspace{3pt}
\item
$\ol{\Lc}(G)=1$.
\end{enumerate}
\end{Theorem}

\begin{proof}
(i)
Denote the reduced tree $\RR(\OO)$ by $\RR$, and set $x=\Phi_{\OO}(f)$ and $H=G_{\OO}$.
We observe that $|x|=|\OO|$ by \cref{induced regular orbit}.
Consider now a vertex $v$ in $\OO$ and an arbitrary element $g\in\Rist_G(v)$, and set $y=\psi_v(g)$
(here $v$ is considered as a vertex in $\T$).
Then $h=\Phi_{\OO}(g)$ lies in $\Rist_H(v)$ and $\psi_v(h)=y$
(here $v$ is considered as a vertex in $\RR$).
By \cref{wreath}, we have $\langle h,x \rangle\cong \langle h \rangle \wr \langle x \rangle$.
Since $x\in\Lc(H)$, \cref{Lieb} implies that both $|y|$ and $|x|$ are $p$-powers for some prime $p$.
Thus $|g|$ and $|\OO|$ are $p$-powers.
Since $g\in\Rist_G(v)$ was arbitrary and $f$ acts transitively on $\OO$, we conclude that
$\Rist_G(\OO)$ is a $p$-group.

(ii)
Let again $f\in\Lc(G)$ be a non-trivial element of finite order.
By applying \cref{f->fi} to $f$, we obtain non-trivial totally splitting $f$-orbits $\OO_1,\ldots,\OO_k$, all lying on the same level $n$ of the tree, such that $|f|=\lcm(|\OO_1|,\ldots,|\OO_k|)$.
Let us fix $i\in\{1,\ldots,k\}$.
By (i), there exists a prime $p$ (in principle, depending on $i$) such that $|\OO_i|$ is a $p$-power and $\Rist_G(\OO_i)$ is a $p$-group.
Since $G$ acts now level transitively on $\T$, all rigid vertex stabilizers are isomorphic by
\eqref{conjugate of rist}.
It follows that $p$ is the same for all $i$ and consequently $\Rist_G(n)$ is a $p$-group.
Also the length of all orbits $\OO_1,\ldots,\OO_k$ is a power of $p$ and $f$ is a $p$-element.

(iii)
By contradiction, assume that $f\in\ol{\Lc}(G)$, $f\ne 1$.
Let $d$ be the Engel degree of $f$.

Assume first that $f$ is of finite order.
Let $\OO$ be a non-trivial totally splitting $f$-orbit.
Define $x$ and $y$ as in the proof of (i), and recall that these are $p$-elements.
By \cref{Lieb},
\[
d \ge |x|+\frac{1}{p}(\log_p |y|-1)(p-1)|x|.
\]
On the other hand, since the exponent of $\Rist_G(n)$ is not finite by \cref{rist infinite exponent}, the order of $y$ is unbounded.
This is a contradiction.

Assume now that the order of $f$ is infinite.
By \cref{lcm of orbit lengths}, there exists an $f$-orbit $\OO$ of length $\ell>d$.
Let once again $\RR$ be the reduced tree $\RR(\OO)$, and set $h=\Phi_{\OO}(f)$ and $H=G_{\OO}$.
Then $h\in S\langle x \rangle$, where $S$ is the first level stabilizer in $\Aut \RR$ (i.e.\ the stabilizer of
$\OO$) and $x$ is a rooted automorphism corresponding to a cycle of length $\ell$.
Observe that $S\langle x \rangle$ is isomorphic to a regular wreath product $W=Y\wr X$, where $Y$ is the stabilizer in $\RR$ of a vertex in $\OO$ and $X=\langle x \rangle$ is cyclic of order $\ell$.
Under this isomorphism, $h$ corresponds to an element $w$ with $\pi(w)=x$.
Also $h$ lies in $\ol{\Lc}_H(D)$ with Engel degree at most $d$, where $D=\Phi_{\OO}(\Rist_G(\OO))$ corresponds to a non-trivial direct product inside the base group of $W$.
Now, by applying (i) of \cref{Comm}, we get $d\ge \ell$, which is a contradiction.
This completes the proof of (iii).
\end{proof}

Now we proceed to prove \cref{BranchTheorem}.

\begin{Theorem}
\label{branch case}
Let $G$ be a branch group.
If $\Lc(G)\ne 1$ then $G$ is periodic and there exists a prime $p$ such that:
\begin{enumerate}
\item
$\Lc(G)$ consists of $p$-elements.
\item
$G$ is virtually a $p$-group.
\end{enumerate}
\end{Theorem}

\begin{proof}
It suffices to show that $\Lc(G)$ does not contain any  elements of infinite order.
Indeed, since $\Lc(G)\ne 1$, the theorem then follows immediately from (ii) of \cref{weakly branch case}, by taking into account that $|G:\Rist_G(n)|$ is always finite if $G$ is a branch group.

Let us assume then that $f\in\Lc(G)$ is of infinite order.
Consider an $f$-orbit $\OO$ in $V(\T)$ of length $\ell\ge 2$, and let $n$ be the level of $\T$ containing $\OO$.
Set $\RR=\RR(\OO)$, $h=\Phi_{\OO}(f)$ and $H=G_{\OO}$.
Then for every vertex $v\ne\emptyset$ of $\RR$ we have $\Phi_{\OO}(\Rist_G(v))\subseteq \Rist_H(v)$, and consequently all rigid vertex stabilizers of $H$ are non-trivial.
Also  $h\in \Lc(H)$.

If $h$ has finite order, then by (i) of \cref{weakly branch case}, the rigid stabilizer in $H$ of some vertex $v\ne\emptyset$ of $\RR$ is periodic.
Consequently $\Rist_G(v)$ is periodic, and by level transitivity of $G$, also
$\Rist_G(n)$ is periodic.
Since $|G:\Rist_G(n)|$ is finite, it follows that $G$ itself is periodic, which is a contradiction.

Assume now that the order of $h$ is infinite.
As in the proof of (iii) of \cref{weakly branch case}, $h$ lies in $S\langle x \rangle$, where $S$ is the first level stabilizer of $\Aut\RR$, and $x$ is a rooted automorphism corresponding to a cycle of length $\ell$.
We can identify $S\langle x \rangle$ with the regular wreath product $W=Y\wr X$, where
$X=\langle x \rangle$ is cyclic of order $\ell$ and $h$ maps onto $x$.
Then $h\in \Lc_W(D)$, where
\[
D=\Phi_{\OO}(\Rist_G(\OO))=\Phi_{\OO}(\Rist_G(n))
\]
corresponds to a non-trivial direct product inside the base group of $W$.
By (ii) of \cref{Comm}, $C_D(h)$ is periodic.
However, since $G$ is branch we have $f^k\in\Rist_G(n)$ for some $k\ge 1$ and then
$h^k=\Phi_{\OO}(f^k)\in D$.
It follows that $h^k\in C_D(h)$ is an element of infinite order, which is a contradiction.
\end{proof}

Now we can apply Theorems A and B to some distinguished subgroups of $\Aut\T$, and obtain the part of Corollary D regarding left Engel elements.
Before proceeding, we will introduce some of the groups that appear in the following result.

First of all, the Hanoi Tower group $\HH$ is the subgroup of $\Aut\T_3$ generated by the three automorphisms $a$, $b$ and $c$ given by the following recursive formulas:
\begin{align*}
a &= (1,1,a) (1\ 2),
\\
b &= (1,b,1) (1\ 3),
\\
c &= (c,1,1) (2\ 3).
\end{align*}
This group models the popular Hanoi Tower puzzle on $3$ pegs.

On the other hand, given an odd prime $p$ and a non-trivial subspace $\mathbf{E}$ of $\F_p^{p-1}$, we define the multi-GGS group (GGS standing for Grigorchuk, Gupta, and Sidki) $G_{\mathbf{E}}$ as the following subgroup of $\Aut\T_p$.
The group $G_{\mathbf{E}}$ is generated by the rooted automorphism $a$ of order $p$ corresponding to the $p$-cycle $(1\ 2\ \ldots \ p)$, and by the elementary abelian $p$-subgroup $B$ consisting of all automorphisms $b_{\mathbf{e}}$, with $\mathbf{e}=(e_1,\ldots,e_{p-1})\in \mathbf{E}$, defined recursively via
\begin{equation}
\label{b_e}
b_{\mathbf{e}} = (a^{e_1},\ldots,a^{e_{p-1}},b_{\mathbf{e}}).
\end{equation}
If $\dim \mathbf{E}=1$ then $G_{\mathbf{E}}$ is simply called a GGS group.
Multi-GGS groups are usually presented by giving a basis $(\mathbf{e}_1,\ldots,\mathbf{e}_r)$ of
$\mathbf{E}$ and defining $b_i\in B$ from $\mathbf{e}_i$ as in \eqref{b_e} for each $i=1,\ldots,r$, so that
$G_{\mathbf{E}}=\langle a,b_1,\ldots,b_r \rangle$.
We refer the reader to the paper \cite{multiggs} by Alexoudas, Klopsch, and Thillaisundaram for general facts about multi-GGS groups.
Multi-GGS groups are infinite and provide a wealth of examples giving a negative answer to the General Burnside Problem.
For instance, the famous Gupta-Sidki $p$-group is the GGS group $\langle a,b \rangle$ with $b$ corresponding to the vector $\mathbf{e}=(1,-1,0,\ldots,0)$.
In general, a multi-GGS group is periodic if and only if $\mathbf{E}$ is contained in the hyperplane of
$\F_p^{p-1}$ given by the equation $e_1+\cdots +e_{p-1}=0$ \cite[Theorem 3.2]{multiggs}.
On the other hand, multi-GGS groups are known to be branch unless
$\mathbf{E}=\langle (1,\ldots,1) \rangle$ consists of constant vectors, in which case it is weakly branch
\cite[Proposition 3.7]{multiggs}.

\begin{Corollary}
In all the following groups, the only left Engel element is the identity:
\begin{enumerate}
\item
Every infinitely iterated wreath product of finite transitive permutation groups of degree at least $2$.
In particular, $\Aut\T$ and $\Gamma_p$, for $p$ a prime.
\item
The group $\FF$ of all finitary automorphisms of $\T$, provided that the sequence $\overline d$ defining $\T$ contains infinitely many terms greater than $2$.
\item
All non-periodic multi-GGS groups $G_{\mathbf{E}}$, i.e.\ those with at least one vector
$\mathbf{e}\in\mathbf{E}$ having non-zero sum in $\F_p$.
\item
The Hanoi Tower group $\HH$.
\end{enumerate}
\end{Corollary}

\begin{proof}
(i)
For every $n\in\N$, let $K_n$ be a finite transitive permutation group of degree $d_n\ge 2$, and let $W$ be the iterated wreath product of all these groups.
Let $\T$ be the spherically homogeneous rooted tree corresponding to the sequence
$\overline d=\{d_n\}_{n\in\N}$.
Then $W$ is isomorphic to the subgroup $K$ of $\Aut\T$ consisting of all automorphisms whose labels at level $n$ are elements of $K_{n+1}$.
Observe that $K$ is a branch group, since every $K_n$ is transitive and obviously $\Rist_K(n)=\St_K(n)$ in this case.

According to Theorem B, we only need to construct an element of infinite order in $K$ to conclude that
$\Lc(W)=1$.
To this purpose, we choose an infinite sequence $\{k_n\}_{n\in\N}$ of non-trivial permutations
$k_n\in K_n$, and an infinite sequence $\{v_n\}_{n\in\N\cup\{0\}}$ of vertices, where $v_n\in\LL_n$ and $k_n(v_n)\ne v_n$.
Also, let $\OO_n$ denote the orbit of $v_n$ under $\langle k_n \rangle$ and set $\ell_n=|\OO_n|$.

Now we define $f$ to be the automorphism of $\T$ having label $k_{n+1}$ at vertex $v_n$ for all
$n\in\N\cup\{0\}$.
We claim that the length of the $f$-orbit of $v_n$ is $\ell_1 \ldots \ell_n$ for all $n\in\N$.
Since $\ell_i\ge 2$ for every $i$, we conclude that $f$ is of infinite order by using (ii) of
\cref{lcm of orbit lengths}.

We prove the claim by induction on $n$.
The result is obvious for $n=1$, since $f$ behaves as $k_1$ on the first level of $\T$.
Then $f^{\ell_1}$ fixes all vertices in the orbit $\OO_1$, and a simple calculation shows that on all those vertices the section of $f^{\ell_1}$ coincides with the section of $f$ at $v_1$, let us call it $g$.
Since $v_n$ lies at level $n-1$ for $g$, by induction the length of the $g$-orbit of $v_n$ is
$\ell_2 \ldots \ell_n$.
From this one can readily see that the $f$-orbit of $v_1$ has length $\ell_1 \ldots \ell_n$, as desired.

(ii)
Obviously, $\FF$ is spherically transitive and $\Rist_{\FF}(n)=\St_{\FF}(n)$ for all $n\in\N$.
Thus $\FF$ is a branch group.
In this case, all elements of $\FF$ are of finite order, but we still get $\Lc(\FF)=1$ from Theorem B, because there is no prime $p$ for which $\FF$ is virtually a $p$-group.
Indeed, assume for a contradiction that $N$ is a normal $p$-subgroup of $\FF$ of finite index $m$.
Under this assumption, if $H$ is a $q$-subgroup of $\FF$ for a prime $q\ne p$, the order of $H$ cannot exceed $m$.
However, as we see in the next paragraph, the condition on the sequence $\overline d$ implies that
$\FF$ has $2$-subgroups and $3$-subgroups of arbitrarily high order, and we get a contradiction.

Consider the following subset of $\N$:
\[
S = \{ n\in\N \mid d_n\ge 3 \}.
\]
By hypothesis, $S$ is infinite.
For every $n\in S$, let $H_n$ be the subgroup of $\FF$ consisting of all automorphisms with labels lying
in $\langle (1\ 2) \rangle$ for all vertices in $\LL_n$ and trivial labels elsewhere.
Then the order of $H_n$ is $2^{\,d_1\ldots d_n}$, which tends to infinity as $n\to\infty$.
We can define similarly a subgroup $J_n$ of order $3^{d_1\ldots d_n}$ for every $n\in S$, by using the $3$-cycle $(1\ 2\ 3)$.
Thus we get $2$-subgroups and $3$-subgroups of $\FF$ of arbitrarily high order, as desired.

(iii)
If $\mathbf{E}=\langle (1,\ldots,1) \rangle$ then $\Lc(G_{\mathbf{E}})=1$ by \cite[Theorem 7]{albert}.
Otherwise $G_{\mathbf{E}}$ is a branch group, and the result follows immediately from Theorem B and from the characterisation of periodic multi-GGS groups given above.

(iv)
The Hanoi Tower group is known to be a branch group \cite[Theorem 5.1]{Hanoi}.
Let us see that the element $ab=(b,1,a)(1\ 2\ 3)$ is of infinite order.
Assume, for a contradiction, that $|ab|=k$ is finite.
Observe that $k=3\ell$ for some $\ell$, since $ab$ has order $3$ modulo the first level stabilizer.
But then
\[
(ab)^{3\ell} = ((ba)^{\ell},(ab)^{\ell},(ab)^{\ell})
\]
implies that $(ab)^{\ell}=1$, which is a contradiction.
\end{proof}

\begin{Corollary}
Let $p$ be a prime and let $\FF_p$ be the group of $p$-finitary automorphisms of $\Aut\T_p$.
Then the following hold:
\begin{enumerate}
\item
$\Lc(\FF_p)=\FF_p$.
\item
$\ol{\Lc}(\FF_p)=1$.
\end{enumerate}
\end{Corollary}

\begin{proof}
Since $\FF_p$ is locally a finite $p$-group, (i) is clear.
On the other hand, since $\FF_p$ is spherically transitive and $\Rist_{\FF_p}(n)=\St_{\FF_p}(n)$ for all $n$, (ii) follows directly from Theorem A.
\end{proof}

\section{Right Engel elements in weakly branch groups}
\label{sec:right engel}

In this final section, we prove Theorem C, regarding right Engel elements in weakly branch groups, and then we apply it to show that $\Rc(G)=1$ whenever $G$ is a GGS group.
Before proceeding, we need a straightforward lemma.

\begin{Lemma}
\label{formula for R(G)}
Suppose that $\T$ has $d$ vertices in the first level, and consider $x,y\in\Aut\T$ such that:
\begin{enumerate}
\item
$y=az$, where $a$ is the rooted automorphism corresponding to the cycle $(1\ 2\ \ldots \ d)$ and $z\in\St(1)$ is given by $\psi(z) = (z_1,\ldots,z_d)$.
\item
$x\in\St(1)$ is given by $\psi(x) = (x_1,\ldots,x_d)$.
\end{enumerate}
Then, for all $k\ge 2$, we have
\[
\psi([y,_k x])
=
([(x_d^{-1})^{z_1},_{k-1} x_1]^{x_1}, \dots, [(x_{d-1}^{-1})^{z_d},_{k-1} x_d]^{x_d}).
\]
\end{Lemma}

\begin{proof}
We have
\begin{align*}
\psi([y, x])
&=
\psi((x^{-1})^{y}x) = \psi((x^{-1})^{a})^{\psi(z)}\psi(x)
\\
&=
((x_d^{-1})^{z_{1}}x_{1}, (x_{1}^{-1})^{z_{2}}x_2, \dots,  (x_{d-1}^{-1})^{z_d}x_d).
\end{align*}
Now the result follows immediately by observing that taking subsequent commutators with $x$ is performed componentwise.
\end{proof}

Now we are ready to prove \cref{theoremrightengel}.

\begin{Theorem}
\label{weakly branch R(G)=1}
Let $G$ be a weakly branch group.
If $\Rist_G(n)$ is not an Engel group for all $n\in\N$, then $\Rc(G)=1$.
\end{Theorem}

\begin{proof}
Let $f\in G$, $f\ne 1$, and assume by way of contradiction that $f\in \Rc(G)$.
Choose a non-trivial $f$-orbit $\OO=\{v_1,\ldots,v_d\}$, and assume that $f$ permutes cyclically the vertices $v_i$.
Let $\RR=\RR_{\OO}$, $H=G_{\OO}$ and $y=\Phi_{\OO}(f)\in\Rc(H)$.
Then we can write $y=az$, where $a$ is rooted in $\RR$ corresponding to the cycle $(1\ 2\ \ldots\ d)$
and $z$ is in the first level stabilizer.
Write $\psi(z)=(z_1,\ldots,z_d)$.

Let $n$ be the level of $\T$ where $\OO$ lies.
Since $\Phi_{\OO}(\Rist_G(n))\subseteq \Rist_H(1)$ and $\Rist_G(n)$ is not Engel by hypothesis, it follows that $\Rist_H(1)$ is not an Engel group.
If $L$ is the first component of the direct product $\psi(\Rist_H(1))$ then $L$ is not an Engel group either, and we can choose $a,b\in L$ such that $[b,_k a]\ne 1$ for all $k\ge 1$.
Now consider $r_1,r_2\in\Rist_H(1)$ such that
\[
\psi(r_1)=(a,1,\ldots,1)
\qquad
\text{and}
\qquad
\psi(r_2)=(b,1,\ldots,1),
\]
and define $x=r_1(r_2^{-1})^{y^{-1}}$, so that
\[
\psi(x) = (a,1,\ldots,1,(b^{-1})^{z_1^{-1}}).
\]
By applying the formula in \cref{formula for R(G)}, we get
\[
\psi([y,_k x]) = ([b,_{k-1} a]^a, *, \dots, *)
\]
and consequently $[y,_k x]\ne 1$ for all $k\ge 2$.
This is a contradiction, since $y\in\Rc(H)$ and $x\in H$.
\end{proof}

Theorem C can be applied to show that GGS groups have no non-trivial right Engel elements.
We first need to prove the weaker result that they are not Engel groups.

\begin{Lemma}
\label{GGS not Engel}
Let $G$ be a GGS group.
Then $G$ is not an Engel group.
\end{Lemma}

\begin{proof}
We show that there is a power of $b$ that is not a left Engel element of $G$.
Let $\mathbf{e}$ be the defining vector of $b$.
Consider any index $i\in\{1,\ldots,p-1\}$ such that $e_{p-i}\ne 0$ in $\F_p$, and choose
$\lambda\in\F_p^{\times}$ such that $\lambda e_{p-i}=-i$.
Then we have
\begin{equation}
\label{psi of b^-1 conjugate}
\psi((b^{-\lambda})^{a^i}) = (\ast,\ldots,\ast,a^i),
\end{equation}
where we use $\ast$ to denote unspecified elements of $G$.

Since $(b^{-\lambda})^{a^i}=[a^i,b^{\lambda}]b^{-\lambda}$, it follows that, for every $k\ge 2$,
\begin{align*}
[(b^{-\lambda})^{a^i},b^{\lambda},\overset{k-1}{\ldots},b^{\lambda}]
&=
[[a^i,b^{\lambda}]b^{-\lambda},b^{\lambda},\overset{k-1}{\ldots},b^{\lambda}]
\\
&=
[[a^i,b^{\lambda},b^{\lambda}]^{b^{-\lambda}},b^{\lambda},\overset{k-2}{\ldots},b^{\lambda}]
\\
&=
[a^i,b^{\lambda},\overset{k}{\ldots},b^{\lambda}]^{b^{-\lambda}}.
\end{align*}
By using \eqref{psi of b^-1 conjugate}, it follows that
\begin{equation}
\label{recursive comm}
\psi([a^i,b^{\lambda},\overset{k}{\ldots},b^{\lambda}]^{b^{-\lambda}})
=
\psi([(b^{-\lambda})^{a^i},b^{\lambda},\overset{k-1}{\ldots},b^{\lambda}])
=
(\ast,\ldots,\ast,[a^i,b^{\lambda},\overset{k-1}{\ldots},b^{\lambda}]).
\end{equation}

Now if $b^{\lambda}$ is a left Engel element of $G$, choose the minimum $k\ge 1$ such that
$[a^i,b^{\lambda},\overset{k}{\ldots},b^{\lambda}]=1$.
Since $a^i$ and $b^{\lambda}$ do not commute, we have $k\ge 2$ and so
$[a^i,b^{\lambda},\overset{k-1}{\ldots},b^{\lambda}]\ne 1$.
According to \eqref{recursive comm}, this is a contradiction.
\end{proof}

\begin{Corollary}
Let $G$ be a GGS group.
Then $\Rc(G)=1$.
\end{Corollary}

\begin{proof}
If the defining vector $\mathbf{e}$ is constant, then $\Lc(G)=1$ by Theorem 7 of \cite{albert}, and consequently also $\Rc(G)=1$.
Thus in the remainder we assume that $\mathbf{e}$ is not constant.
By Lemmas 3.2 and 3.4 of \cite{Gustavo}, we know that $G$ is regular branch over $K$, where $K=\gamma_3(G)$ if $\mathbf{e}$ is symmetric and $K=G'$ otherwise.
Since $\Rist_G(n)$ contains a copy of $K\times \overset{p^n}{\cdots} \times K$ for every $n\in\N$, if we prove that $K$ is not Engel then Theorem C applies to conclude that $\Rc(G)=1$.

In order to show that $K$ is not Engel, we are going to find a vertex $v$ of the first level of the tree such that $\psi_v(K)=G$.
Since $G$ is not Engel by \cref{GGS not Engel}, it follows that $K$ is not Engel either, as desired.

We consider separately the cases when $\mathbf{e}$ is symmetric and non-symmetric.
Assume first that $\mathbf{e}$ is non-symmetric, so that $K=G'$.
We have
\[
\psi([b,a])
=
(a^{-e_1}b,a^{e_1-e_2},a^{e_2-e_3},\ldots,a^{e_{p-2}-e_{p-1}},b^{-1}a^{e_{p-1}}).
\]
Since $\mathbf{e}$ is not constant, there exists $i\in\{1,\ldots,p-2\}$ such that $e_i\ne e_{i+1}$ in $\F_p$.
If $v$ is the vertex $i+1$ on the first level of the tree, then
\[
\psi_v([b,a]) = a^{e_i-e_{i+1}}
\qquad
\text{and}
\qquad
\psi_v([b,a]^{a^i}) = a^{-e_1}b.
\]
Since the subgroup $\langle a^{e_i-e_{i+1}}, a^{-e_1}b \rangle$ coincides with $G$, we get the desired equality $\psi_v(G')=G$.

Now let $\mathbf{e}$ be symmetric, i.e.\ such that $e_i=e_{p-i}$ for all $i=1,\ldots,p-1$.
Since $\mathbf{e}$ is not constant, this implies that $p\ge 5$.
We have
\[
\begin{split}
\psi([b,a,a])
&=
\psi([b,a]^{-1}) \psi([b,a]^a)
\\
&=
(b^{-1}a^{e_1}b^{-1}a^{e_{p-1}}, a^{-2e_1+e_2}b, a^{e_1-2e_2+e_3}, \ldots,
\\
&\qquad\qquad\qquad\qquad\qquad\qquad\qquad
a^{e_{p-3}-2e_{p-2}+e_{p-1}}, a^{-e_{p-1}}ba^{e_{p-2}-e_{p-1}}).
\end{split}
\]
If $e_i-2e_{i+1}+e_{i+2}\ne 0$ for some $i\in\{1,\ldots,p-3\}$, we have a non-trivial power of $a$ in one of the components of $\psi([b,a,a])$ and we can argue as above to prove that $\psi_v(\gamma_3(G))=G$ for a vertex $v$ in the first level.
On the other hand, if $e_i-2e_{i+1}+e_{i+2}=0$ for all $i=1,\ldots,p-3$, then
\[
\begin{split}
e_{3}&=2e_{2}-e_{1},
\\
e_{4}&=2e_{3}-e_{2}=3e_{2}-2e_{1},
\\
&\,\,\,\vdots
\\
e_{p-1}&=2e_{p-2}-e_{p-3}=(p-2)e_{2}-(p-3)e_{1}.
\end{split}
\]
Since $e_{p-1}=e_1$, the last equation implies that $e_1=e_2$, and then using all other equations, we get that all components $e_i$ are equal to $e_1$.
Thus the vector $\mathbf{e}$ is constant, which is a contradiction.
\end{proof}

\subsection*{Acknowledgements}
The authors want to thank A.\ Tortora and G.\ Traustason for helpful discussions.

\bibliography{bib}

\end{document}